\documentclass[12pt,a4paper]{amsart}

\usepackage{latexsym} 

\usepackage[dvips]{graphics}
\usepackage{epsfig}

\usepackage{amssymb}
\usepackage{amsthm}
\usepackage{amsfonts}
\usepackage{amsmath}
\usepackage{amstext}
\usepackage{amscd}

\usepackage{epic}
\usepackage{eepic}
\usepackage{pstricks,pst-plot}

\numberwithin{equation}{section}

\theoremstyle{plain}
\newtheorem{thm}{Theorem}[section] 

\newtheorem{cor}[thm]{Corollary}
\newtheorem{lem}[thm]{Lemma}
\newtheorem{ass}[thm]{Assertion}

\newtheorem{theorem*}{Theorem}[]
\newtheorem{prope}[thm]{Property}

\theoremstyle{definition}

\newtheorem{notation}{Notation}
\newtheorem{example}[thm]{Example}
\newtheorem{obs}[thm]{Observation}

\theoremstyle{remark}
\newtheorem{rem}[thm]{Remark}

\theoremstyle{note}
\newtheorem{note}[thm]{Note}

\newcommand{\R}{\mathbb{R}}

\def\accentclass@{7}
\def\makeacc@#1#2{\def#1{\mathaccent"\accentclass@#2 }}
\makeacc@\cir{017}

\date{\today}

\begin{document}

\newcommand{\COMP}{\raisebox{0.1ex}{\scriptsize $\circ$}}

\title[Finiteness theorems]
{On finiteness theorems of polynomial functions} 
\author{Satoshi KOIKE and Laurentiu Paunescu}

\dedicatory{Dedicated to Takuo Fukuda for his 80th birthday}

\address {Department of Mathematics, Hyogo University
of Teacher Education, 942-1 Shimokume, Kato,
Hyogo 673-1494, Japan}
\email {koike@hyogo-u.ac.jp}

\address{School of Mathematics and Statistics, University of Sydney, Sydney, 
NSW, 2006, Australia}
\email{laurentiu.paunescu@sydney.edu.au}

\subjclass[2010]{Primary 14P10, 32S15 Secondary 32S45, 57R45}

\keywords{polynomial function, Nash function, topological equivalence, 
semialgebraic equivalence}

\thanks{This research is partially supported by  
JSPS KAKENHI Grant Number 26287011, 20K03611. }

\newcommand{\abstracttext}{}

\maketitle


\begin{abstract}
Let $d$ be a positive integer. 
We  show a finiteness theorem for semialgebraic $\mathcal{RL}$ 
triviality of a Nash family of Nash functions defined on a Nash 
manifold, generalising Benedetti-Shiota's  finiteness theorem 
for semialgebraic $\mathcal{RL}$ equivalence classes appearing in the 
space of real polynomial functions of degree not exceeding $d$. 
We also prove  Fukuda's claim, Theorem \ref{localF}, and its semialgebraic version Theorem 
\ref{localBS}, on the finiteness of the local $\mathcal R$ types appearing in the space 
of real polynomial functions of real polynomial function germs 
of degree not exceeding $d$.

%
\end{abstract} 

\vspace{4mm} 


\section{Introduction.} 

Let $d$ be a positive integer, and 
let us denote by $P(n,1,d;\R )$ the set of  polynomial functions 
from $\R^n$ to $\R$ 
of degree not exceeding $d$. 
The space $P(n,1,d;\R )$ is identified with the Euclidean space $\R^{L+1},$
regarded as  the coefficient space of all polynomials in $P(n,1,d;\R )$. 

We say that two polynomials $f, g : \R^n \to \R$ 
are {\em topologically $\mathcal{RL}$ equivalent} (respectively {\em topologically $\mathcal{R}$ equivalent}), if there exist 
homeomorphisms $\sigma : \R^n \to \R^n$ and $\tau : \R \to \R$ such that 
$\tau \circ f = g \circ \sigma$ (respectively $ f = g \circ \sigma$). 
Furthermore we say that two polynomials 
$f, g : \R^n \to \R$ are {\em semialgebraically $\mathcal{RL}$ equivalent} (respectively {\em semialgebraically  $\mathcal{R}$ equivalent})
if we can take the $\sigma$ and $\tau$ as semialgebraic homeomorphisms (respectively $\sigma$ semialgebraic). 

In \cite{fukuda1} T. Fukuda established the following finiteness theorem
for polynomials in $P(n,1,d;\R )$. 

\begin{thm}\label{fukudafinite}(Fukuda's finiteness theorem) 
The number of topological $\mathcal{RL}$ types appearing in $P(n,1,d;\R )$ 
is finite. 
\end{thm} 

On the other hand, R. Benedetti and M. Shiota strengthened 
Fukuda's finiteness theorem in \cite{benedettishiota} as follows. 

\begin{thm}\label{benedettishiota}(Benedetti-Shiota's finiteness theorem) 
The number of semialgebraic $\mathcal{RL}$ types appearing in $P(n,1,d;\R )$ 
is finite. 
\end{thm} 

Based on the above finiteness theorems and their proofs, 
we discuss finiteness properties for $\mathcal{RL}$ or $\mathcal{R}$ 
equivalence of polynomials in \S 4.
In this introduction we mention two local results from those 
finiteness properties. 
Let us denote  the polynomials without free term by $P_0(n,1,d;\R )=\{f\in P(n,1,d;\R)|f(0)=0\},$  
naturally identified with the Euclidean space 
$\R^{L}$.
We say that two polynomials $f, g \in P(n,1,d;\R )$ 
are {\em locally topologically $\mathcal{R}$ equivalent at $0\in\R^n$, and write $\mathcal{R}_0$}, if there exists 
a local homeomorphism $\sigma : (U\subset \R^n,0) \to (V\subset \R^n,0)$ between two open neighbourhoods of the origin,  such that 
$f_{/U}= g_{/V}\circ \sigma$, and that two polynomials
$f, g \in P(n,1,d;\R )$ are 
{\em locally semialgebraically $\mathcal{R}$ equivalent at $0\in\R^n$, and write $\mathcal{R}_0$},
if we can take the $\sigma$ as a local semialgebraic homeomorphism.
Then we will prove the following new finiteness results for the local $\mathcal{R}_0$ 
euivalence of polynomial functions.

\begin{thm}\label{localF}
(Fukuda's local $\mathcal{R}$ finiteness theorem) 
(\cite{fukuda3})
The number of local topological $\mathcal{R}_0$ types 
appearing in $P_0(n,1,d;\R )$ is finite.
In addition, each topological $\mathcal{R}_0$ equivalence class in
$P_0(n,1,d;\R )$ is a semialgebraic subset of $P_0(n,1,d;\R )$.
\end{thm}

\begin{thm}\label{localBS}
(local semialgebraic $\mathcal{R}_0$ finiteness theorem)
The number of local semialgebraic $\mathcal{R}_0$ types 
appearing in $P_0(n,1,d;\R )$ is finite.
\end{thm}

As a generalisation of the above finiteness results 
(Theorems \ref{fukudafinite} and \ref{benedettishiota})
for topological or semialgebraic $\mathcal{RL}$ equivalence classes 
appearing in the space of real polynomial functions of a fixed degree, 
we shall show the following result in \S 5.

\begin{thm}\label{finitenessNF}
Let $\{ f_t : M \to \R \ | \ t \in J \}$ be a Nash family of Nash 
functions defined on a Nash manifold $M$ with a semialgebraic 
parameter space $J$.
Namely, the function $F : M \times J \to \R$ defind by 
$F(x,t) := f_t(x)$ is a Nash function. 
Then there exists a finite subdivision of $J$ into Nash open simplices 
$J = Q_1 \cup \cdots \cup Q_u$ such that 
$$
\{ f_t : M \to \R \ | \ t \in Q_i \} 
$$
is semialgebraically $\mathcal{RL}$ trivial over each $Q_i$. 
\end{thm}

\noindent For the definition of a Nash open simplex, see \S 3.3.

For  the reader convenience, we  recall  several classic  notions. Firstly, in \S 2, that of  Whitney's  regularity and of 
stratified mappings, followed by the notion of  controlled tube systems for  stratifications. 
At the end of this section we introduce the semialgebraic and Nash properties 
which will be used to prove  our results.
Secondly,  in  \S 3, we recall Thom's 1st and 2nd Isotopy Lemmas and Shiota's semialgebraic 
versions of Thom's Isotopy Lemmas as tools to establish topological 
$\mathcal{RL}$ triviality and semialgebraic $\mathcal{RL}$ triviality, respectively.


\bigskip
\section{Preliminaries.}
\label{SPTIL}


\subsection{Whitney-regularity and stratified mapping}\label{Whitneyregular}

Let $X, \ Y$ be disjoint $C^1$ submanifolds of $\R^m$, and 
let $y_0$ be a point in $Y \cap \overline{X}$.
We say that the pair $(X,Y)$ satisfies {\em Whitney's condition $(b)$
at} $y_0$, if for any sequences $\{ x_i \}$ in $X$ and $\{ y_i \}$ 
in $Y$ which tend to $y_0$ such that the tangent planes $T_{x_i}X$ tend to some $\dim X$-plane
$\tau \subset \R^n$ and the secants $\widehat{x_i y_i}$ tend (in the projective 
space $\mathbb{P}^{m-1}$) to some line $\ell$, then we have 
$\ell \subset \tau$.
If $(X,Y)$ satisfies Whitney's condition $(b)$ at any point of $Y$, 
we say that $X$ is {\em Whitney $(b)$-regular over} $Y$.
For the properties of the Whitney regularity, see 
H. Whitney \cite{whitney1, whitney2}, J. N. Mather \cite{mather} and 
D. J. A. Trotman \cite{trotman1, trotman2}.

Let $A \subset \R^m$ admit a $C^1$ stratification ${\mathcal S}(A)$ 
satisfying the {\em frontier condition}, namely if $\overline{X} \cap Y 
\ne \emptyset$ for $X,\ Y \in {\mathcal S}(A)$, then we have 
$\overline{X} \supset Y$.  
We say that the stratification ${\mathcal S}(A)$ of $A$ is a {\em Whitney 
stratification}, if for any strata $X,\ Y \in {\mathcal S}(A)$ with 
$\overline{X} \supset Y,\ X$ is Whitney ($b$)-regular over $Y$.
In the sequel  we will assume the frontier condition whenever we consider 
stratifications of  subsets  in Euclidean spaces.

Let $X, \ Y$ be disjoint $C^1$ submanifolds of $\R^m$, and
let $y_0$ be a point in $Y \cap \overline{X}$.
Let $f : X\cup Y \to \R^r$ be a  mapping such that 
both  restricted mappings $f|_X : X \to \R^r$ and 
$f|_Y : Y \to \R^r$ are $C^1$ and of constant ranks. 
We say that the pair $(X,Y)$ satisfies {\em Thom's condition $(a_f)$
at} $y_0$, if for any sequence $\{ x_i \}$ in $X$ tending to $y_0$ 
such that $ker \ d(f|_X)_{x_i}$ tends to $\kappa$, then we have 
$ker \ d(f|_Y)_{y_0} \subset \kappa$, where $ker \ d(f|X)_x$ denotes 
the kernel of the differentiable of $f|_X$ at $x$. 
If $(X,Y)$ satisfies Thom's condition $(a_f)$ at any point of $Y$, 
we say that $X$ is {\em Thom $(a_f)$-regular over} $Y$ 
or {\em Thom regular over} $Y$. 
For the properties of the Thom $(a_f)$ regularity, see 
\cite{gibsonetal, koike1, mather, thom}.

Let $A \subset \R^m$ and $B \subset \R^r$ and let $f : A \to B$ be 
a mapping.
Assume that $A$ and $B$ admit $C^1$ stratifications ${\mathcal S}(A)$ and 
${\mathcal S}(B)$ respectively.
We call $f$ a {\em stratified mapping}, if for any stratum 
$X \in {\mathcal S}(A)$, there exists a stratum $U \in {\mathcal S}(B)$ 
such that $f|_X : X \to U$ is a surjective  $C^1$ submersion. 
Let us remark that we are assuming also ``surjectiveness" for a stratified mapping 
in this paper.
We call the stratified mapping 
$f : (A,\mathcal{S}(A)) \to (B,\mathcal{S}(B))$ {\em Thom regular},
if for any $X, Y \in \mathcal{S}(A)$ with $\overline{X} \supset Y$, 
$X$ is Thom $(a_f)$-regular over $Y$.

\subsection{Controlled tube system}\label{CTS}
Let $M$ be a $C^2$ manifold, and let $N$ be a $C^2$ submanifold of $M$. 
A {\em tube} at $N$ in $M$ is a triple $T = (|T|,\pi , \rho )$, 
where $|T|$ is an open neighbourhood of $N$ in $M$, 
$\pi : |T| \to N$ is a submersive $C^2$ retraction, 
and $\rho$ is a nonnegative $C^2$ function on $|T|$ such that 
$\rho^{-1}(0) = N$ and each point $x$ of $N$ is the unique 
critical point (nondegenerate) of the restriction of $\rho$ to 
$\pi^{-1}(x)$.
In the case when $N$ is an open submanifold of $M$, we
let $T = (|T|,\pi, \rho )$ where $|T| = N$, $\pi = id_N$ 
and $\rho = 0$ for convenience.
Here $id_N$ denotes the identity mapping $id_N :N \to N$. 

Let $A \subset M$, and let $\mathcal{S}(A) = \{ X_i \}$ be a $C^2$ 
stratification of $A$ (the strata are $C^2$ submanifolds). 
A {\em tube system} for the stratification $\{ X_i \}$ consists of one tube 
$T_i = (|T_i|,\pi_i,\rho_i )$ at each $X_i$. 
We say that a tube system $\{ T_i\}$ is {\em controlled} if for any $i$ and 
$j$ with $X_i \subset \overline{X_j}$, the following commutative 
relations hold: 
$$
\pi_i \circ \pi_j (x) = \pi_i (x), \ \rho_i \circ \pi_j (x) = \rho_i (x) \ 
\text{for} \ x \in |T_i| \cap |T_j|.
$$
We call $\{ T_i \}$ {\em weakly controlled} if the first relation holds.


Let $M$ and $P$ be $C^2$ manifolds, let $f : A \to B$ be a 
mapping between $A \subset M$ and $B \subset P$. Let 
$\mathcal{S}(A) = \{ W_k\}$ and $\mathcal{S}(B) = \{ V_j^{\prime}\}$ be 
$C^2$ stratifications, and let $\{ T_k = (|T_k|,\pi_k ,\rho_k )\}$ and 
$\{ T_j^{\prime} = (|T_j|^{\prime},\pi_j^{\prime},\rho_j^{\prime})\}$
be tube systems for $\{ W_k\}$ and $\{ V_j^{\prime}\}$, respectively. 
We call $\{ T_k\}$ {\em controlled over} $\{ T_j^{\prime}\}$ if $\{ T_k\}$ 
satisfies the following relations: 

(1) For any $k$ and $\ell$ with $W_k \subset \overline{W_{\ell}}$, 
$\pi_k \circ \pi_{\ell} (x) = \pi_k (x)$ on $|T_k| \cap |T_{\ell}|$. 

(2) For any $k$ and $\ell$ with $W_k \subset \overline{W_{\ell}}$ and 
$f(W_k) \cup f(W_{\ell}) \subset V_j^{\prime}$ for some $j$, 
$\rho_k \circ \pi_{\ell} (x) = \rho_k(x)$ on $|T_k| \cap |T_{\ell}|$. 

(3) For any $k$ and $j$ with $f(W_k) \subset V_j^{\prime}$, $f|_{W_k}$ of class  $C^2$ and
$f(|T_k|) \subset |T_j^{\prime}|$ and $f \circ \pi_k = \pi_j^{\prime} \circ f$ 
on $A \cap |T_k|$.

\subsection{Semialgebraic and Nash properties.}\label{saproperty}
In this subsection we recall some important semialgebraic and Nash 
properties which can be useful tools when we apply Thom's Isotopy Lemmas 
to a family of polynomial functions or more generally to a family of 
Nash functions.
We first recall the notions of a Nash manifold and of a Nash mapping.

Let $s = 1, 2, \cdots , \infty, \omega.$ 
A $C^s$ submanifold of $\R^m$ is called a $C^s$ {\em Nash manifold}, if 
it is semialgebraic in $\R^m$. 
In this paper, a submanifold always means a regular submanifold. 
By B. Malgrange \cite{malgrange}, a $C^{\infty}$ Nash manifold is 
a $C^{\omega}$ manifold.
We simply call a $C^{\infty}$ Nash manifold a {\em Nash manifold}.

Let $M \subset \R^m$ and $N \subset \R^n$ be $C^s$ Nash manifolds.
A $C^r$ mapping $f : M \to N$, $r \le s$, is called 
a $C^r$ {\em Nash mapping}, if the graph of $f$ is semialgebraic 
in $\R^m \times \R^n$.
It is known also by B. Malgrange \cite{malgrange} that a $C^{\infty}$ 
Nash mapping is a $C^{\omega}$ mapping.
We simply call a $C^{\infty}$ Nash mapping a {\em Nash mapping}.

Let $A \subset \R^m$ and $B \subset \R^n$ be semialgebraic sets.
A continuous mapping $f : A \to B$ is called 
a {\em semialgebraic mapping}, if the graph of $f$ is semialgebraic 
in $\R^m \times \R^n$.
Note that we demand the continuity for a semialgebraic mapping 
in this paper.

\begin{thm}\label{tarski-seidenberg}
(Tarski-Seidenberg Theorem \cite{seidenberg}).
Let $A$ be a semialgebraic set in $\R^k$,
and let $f : \R^k \to \R^s$ be a semialgebraic mapping.
Then $f(A)$ is semialgebraic in $\R^s$.
\end{thm}

Throughout this paper a {\em Nash open simplex} means a Nash manifold 
which is Nash diffeomorphic to an open simplex in some Euclidean space.

\begin{thm}\label{lojasiewicz}
(Lojasiewicz's Semialgebraic Triangulation Theorem
\cite{lojasiewicz1, lojasiewicz2})
Given a finite system of bounded semialgebraic sets $X_{\alpha}$
in $\R^q$, there exist a simplicial decomposition $\R^q = \cup_a C_a$ 
and a semialgebraic automorphism $\tau$ of $\R^q$ such that

(1) each $X_{\alpha}$ is a finite union of some of the $\tau (C_a)$, 
 
(2) $\tau (C_a)$ is a Nash manifold in $\R^q$ and $\tau$ induces 
a Nash diffeomorphism $C_a \to \tau (C_a)$ for every $a$, 
in other words, $\tau (C_a)$ is a Nash open simplex for every $a$. 
\end{thm}

\begin{rem}\label{remark21}
(1) Concerning the proof of Semialgebraic Triangulation Theorem,
see also H. Hironaka \cite{hironaka2} or M. Coste \cite{coste}. 

(2) There is a Nash embedding of $\R^q$ into $\R^{q+1}$ via 
$S^q \subset \R^{q+1}.$ 
Let us remark that a semialgebraic subset of $\R^q$ can be regarded as 
a bounded semialgebraic subset of $\R^{q+1}$.
\end{rem}

The existence of a finite Whitney stratification of Nash class
for a semialgebraic set was established by S. Lojasiewicz 
(\cite{lojasiewicz1, lojasiewicz2}).

We next recall Fukuda's Lemma in \cite{fukuda1}, concerning the existence 
of compatible Whitney stratifications related to a Nash stratified mapping.

\begin{lem}\label{stratifiedmapping}
(Fukuda's Lemma).
Let $M$ and $N$ be Nash manifolds, and let $f : M \to N$
be a Nash mapping.
Given semialgebraic subsets $A_1, \cdots , A_a$ of $M$ 
and semialgebraic subsets $B_1, \cdots , B_b$ of $N$,
there exist finite Whitney stratifications of Nash class
$\mathcal{S}(M)$ of $M$ compatible with $A_1, \cdots , A_a$
and $\mathcal{S}(N)$ of $N$ compatible with $B_1, \cdots , B_b$
such that $f : (M,\mathcal{S}(M)) \to (N,\mathcal{S}(N))$
is a Nash stratified mapping.
\end{lem}


\bigskip
\section{Thom's Isotopy Lemmas}\label{TIL}

Thom's Isotopy Lemmas have been used widely in  singularity theory, e.g. 
to establish the topological stability theory, to show the topological 
triviality of a family of sets or mappings, to solve the finite $C^0$ 
determinacy problem of map-germs, and so on.
Therefore they are not always given in the same form.
In this section we mention Thom's Isotopy Lemmas and their semialgebraic 
versions in a suitable form for our purposes. 


\subsection{Thom's 1st and 2nd Isotopy Lemmas}\label{T1&2IL}

Thom's 1st and 2nd Isotopy Lemmas were presented by Ren\'e Thom \cite{thom} 
as very useful tools in  singularity theory as mentioned above.. 
For the detailed proofs, see J. N. Mather \cite{mather} or
C. G. Gibson, K. Wirthm\"uller, A. du Plessis and E. J. N. Looijenga 
\cite{gibsonetal}.
We also mention  T. Fukuda \cite{fukuda1} and A. du Plessis and C. T. C. Wall 
\cite{duplessiswall} as further references. 

Now let us recall Thom's Isotopy Lemmas. 
Let $M$, $N$ and $I$ be $C^{\infty}$ manifolds, let $f : M \to N$ and 
$q : N \to I$ be $C^{\infty}$ mappings, and let $A \subset M$ and 
$B \subset N$ be closed subsets such that $f(A) \subset B$.

\begin{thm}\label{1stisotopy}
(Thom's 1st Isotopy Lemma) 
Let $q |_B : B \to I$ be proper. 
Suppose that $B$ admits a locally finite 
$C^{\infty}$ Whitney stratification $\mathcal{S}(B)$ such that 
$q : (B,\mathcal{S}(B)) \to (I,\{ I \} )$ is a stratified mapping.
Then the stratified set $(B,\mathcal{S}(B))$ is locally $C^0$ trivial over $I$.
\end{thm}

\begin{thm}\label{2ndisotopy}
(Thom's 2nd Isotopy Lemma) 
Let $I$ be connected, and $f |_A : A \to N$ and $q |_B : B \to I$ are proper.
Suppose that $A$ and $B$ admit locally finite $C^{\infty}$ Whitney 
stratifications $\mathcal{S}(A)$ and $\mathcal{S}(B)$ respectively 
such that $f : (A,\mathcal{S}(A)) \to (B,\mathcal{S}(B))$ is 
a Thom regular stratified mapping and 
$q : (B,\mathcal{S}(B)) \to (I,\{ I \} )$ is a stratified mapping.
Then the stratified mapping $f$ is topologically trivial over $I$,
namely, there are homeomorphisms
$H : (q \circ f)^{-1}(P_0) \times I \to A$ and
$h : q^{-1}(P_0) \times I \to B$, for some $P_0 \in I$,
preserving the natural stratifications such that
$h^{-1} \circ f \circ H = f|_{(q \circ f)^{-1}(P_0)} \times id_{I}$
and $q \circ h : q^{-1}(P_0) \times I \to I$ 
is the canonical projection, where $id_I$ is the identity map on $I$.

\end{thm}


\subsection{Semialgebraic versions of Thom's Isotopy Lemmas.}\label{SemiTIL}

M. Shiota established in \cite{shiota2} the semialgebraic versions of 
Thom's Isotopy Lemmas.
We recall them in this subsection.

Let $A \subset \R^a,\ B \subset \R^b$ be semialgebraic subsets, 
let $I$ be a Nash open simplex,
and let $f : A \to B$ and $q : B \to I$ be proper Nash mappings.

\begin{thm}\label{s1stisotopy}
(Shiota's semialgebraic version of Thom's 1st Isotopy Lemma) 
Suppose that $B$ admits a finite Whitney stratification of Nash class 
$\mathcal{S}(B)$ such that $q : (B,\mathcal{S}(B)) \to (I,\{ I \} )$ 
is a stratified mapping.
Then the stratified set $(B,\mathcal{S}(B))$ is
semialgebraically trivial over $I$.
\end{thm}

\begin{thm}\label{s2ndisotopy}
(Shiota's semialgebraic version of Thom's 2nd Isotopy Lemma) 
Suppose that $A$ and $B$ admit finite Whitney stratifications of Nash 
class $\mathcal{S}(A)$ and $\mathcal{S}(B)$ respectively such that
$f : (A,\mathcal{S}(A)) \to (B,\mathcal{S}(B))$ is a Thom regular
stratified mapping and $q : (B,\mathcal{S}(B)) \to (I,\{ I \} )$ 
is a stratified mapping.
Then the stratified mapping $f$ is semialgebraically trivial over $I$,
namely, there are semialgebraic homeomorphisms
$H : (q \circ f)^{-1}(P_0) \times I \to A$ and
$h : q^{-1}(P_0) \times I \to B$, for some $P_0 \in I$,
preserving the natural stratifications such that
$h^{-1} \circ f \circ H = f|_{(q \circ f)^{-1}(P_0)} \times id_{I}$
and $q \circ h : q^{-1}(P_0) \times I \to I$ 
is the canonical projection.
\end{thm}


\subsection{Semialgebraic triviality after finite subdivision}\label{SATAFS}

Let $A \subset \R^a$, $B \subset \R^b$ be semialgebraic sets, 
let $Q$ be a Nash open simplex and let $f : A \to B$ and $q : B \to Q$ 
be proper Nash mappings.
Let us assume that $A$ and $B$ admit finite Whitney stratifications of 
Nash class $\mathcal{S}(A)$ and $\mathcal{S}(B)$, respectively. 
In addition, we assume that  $f : (A,\mathcal{S}(A)) \to (B,\mathcal{S}(B))$ 
and $q : (B,\mathcal{S}(B)) \to (Q,\{ Q\})$ are stratified mappings.
Under this setting, we prepare some notations.

\begin{notation}\label{notations}
For $t \in Q$, let $A_t$ and $B_t$ denote $(q \circ f)^{-1}(t) \cap A$ and 
$q^{-1}(t) \cap B$, respectively, and let
$$
\mathcal{S}(A)_t :=\{ (q \circ f)^{-1}(t) \cap  X \ | \ X\in \mathcal{S}(A)\} 
\ \ \text{and} \ \ 
\mathcal{S}(B)_t := \{q^{-1}(t) \cap Y \ | \ Y\in \mathcal{S}(B)\}. 
$$
\end{notation}

Note that for any $t \in Q$, $\mathcal{S}(A)_t$ and $\mathcal{S}(B)_t$ 
are stratifications of $A_t$ and $B_t$, respectively, and 
$f_t := f |_{A_t} : (A_t,\mathcal{S}(A)_t) \to (B_t,\mathcal{S}(B)_t)$ 
is a stratified mapping.
Then, using the same argument as in the proof of  Proposition 4.5 in \cite{koike2} 
together with Theorem \ref{lojasiewicz}, we have the following lemma.

\begin{lem}\label{SATlemma}
Suppose that, for any $t \in Q$, the stratified mapping 
$f_t := f |_{A_t} : (A_t,\mathcal{S}(A)_t) \to (B_t,\mathcal{S}(B)_t)$ 
is $(a_{f_t})$-regular.
Then, subdividing $Q$ into finitely many Nash open simplices 
if necessary, the stratified mapping 
$f : (A,\mathcal{S}(A)) \to (B,\mathcal{S}(B))$ is semialgebraically 
trivial over $Q$ (in the sense of Theorem \ref{s2ndisotopy}).
\end{lem}


\bigskip
\section{Finiteness for some equivalences of polynomial functions}
\label{finitepoly}


\subsection{Finiteness for $\mathcal{R}_0\mathcal{L}_0$-equivalence of polynomial functions.}\label{keeping0}

In this subsection we consider some restricted versions of Theorems \ref{fukudafinite} and 
\ref{benedettishiota}.  
Let us recall $P_0(n,1,d;\R ) := \{ f \in P(n,1,d;\R ) | f(0) = 0\}$. 
We say that two polynomials $f, g \in P(n,1,d;\R )$ 
are {\em topologically $\mathcal{R}_0\mathcal{L}_0$ equivalent} (respectively {\em semialgebraically 
$\mathcal{R}_0\mathcal{L}_0$ equivalent}), if there exist homeomorphisms (respectively semialgebraic 
homeomorphisms) $\sigma : \R^n \to \R^n$ with $\sigma (0) = 0$ and $\tau : \R \to \R$ 
with $\tau (0) = 0$ such that $\tau \circ f = g \circ \sigma$. 

If we apply the same argument as in  the proof of Theorem \ref{fukudafinite} 
we can show that the number of topological $\mathcal{RL}$ types appearing in 
$P_0(n,1,d;\R )$ is finite.
Then, by 
taking the stratifications $\mathcal{S}_1$ and $\mathcal{S}_2$ of 
$\R^n \times P_0(n,1d;\R )$ and $\R \times P_0(n,1,d;\R )$ 
so that $\mathcal{S}_1$ and $\mathcal{S}_2$ are compatible with 
$\{ 0 \} \times P_0(n,1,d;\R )$ where $0 \in \R^n$ and 
$\{ 0 \} \times P_0(n,1,d;\R )$ where $0 \in \R$, respectively, we conclude that
$\sigma$ preserves $0 \in \R^n$ and $\tau$ preserves $0 \in \R$. 
In fact when restricted to $P_0(n,1,d;\R )$, the fact that $\tau$ preserves $0 \in \R$ is a consequence of the fact that $\sigma$ preserves $0 \in \R^n$.  Thus we have shown  the following stronger result. 

\begin{thm}\label{topkeep0}
The number of topological $\mathcal{R}_0\mathcal{L}_0$ types appearing in 
$P_0(n,1,d;\R )$ is finite.
\end{thm}

Using a similar argument to the proof of the above theorem with the semialgebraic versions 
of Thom's Isotopy Lemmas, we can show the following theorem in the semialgebraic case.
 
\begin{thm}\label{semikeep0}
The number of semialgebraic $\mathcal{R}_0\mathcal{L}_0$ types appearing in 
$P_0(n,1,d;\R )$ is finite.
\end{thm}


\subsection{Finiteness for $\mathcal{RL}$-equivalence of polynomial function 
germs}\label{localpoly1}

We say that two polynomials 
$f, g : \R^n \to \R$ are {\em locally  topologically $\mathcal{RL}$ equivalent} at $0 \in \R^n$ and write $\mathcal{RL}_0$, 
if there exist local homeomorphisms $\sigma : (U\subset \R^n,0) \to (V \subset \R^n,0), U, V$  two open neighbourhoods of $0\in \R^n$,
and $\tau : (J\subset \R,f(0)) \to (K\subset \R,g(0)), J, K$ two open intervals,  such that for their appropriate restrictions we have
$\tau \circ f = g \circ \sigma$ .
Similarly,  two polynomials
$f, g : \R^n, \to \R$ are 
{\em locally semialgebraically $\mathcal{RL}$ equivalent} at $0 \in \R^n$
if we can take the $\sigma$ and $\tau$ as local semialgebraic homeomorphisms, and write $\mathcal{RL}_0$.

In the case of $f, g \in P_0(n,1,d;\R)$  the above equivalence relation $\mathcal{RL}_0$, 
is in fact the local version of the previous $\mathcal{R}_0\mathcal{L}_0$ equivalence relation, 
i.e.  $\mathcal{RL}_0$  coincides with 
$\mathcal{R}_0\mathcal{L}_0$ in this case.
In this paper local will mean only local at $0\in\R^n$ and at the corresponding images, and the meaning of local will be as described   above,  i.e. equivalence of germs of polynomial functions at $0\in \R^n$.
.

In the original version of \cite{fukuda1}, T. Fukuda proved the following 
finiteness theorem for local topological $\mathcal{RL}$ equivalence classes 
appearing in the $P_0(n,1,d;\R )$. 
This theorem is an immediate  consequence of 
Theorem \ref{topkeep0}.

\begin{thm}\label{fukudalofinite0} 
The number of local topological $\mathcal{R}_0\mathcal{L}_0$  types appearing in 
$P_0(n,1,d;\R )$ is finite. 
\end{thm}

As a corollary of Theorem \ref{fukudalofinite0}, we have the following.

\begin{cor}\label{fukudalofinite1}
(Fukuda's local $\mathcal{RL}_0$ finiteness theorem) 
The number of local topological $\mathcal{RL}_0$ types appearing in 
$P(n,1,d;\R )$ is finite. 
\end{cor}

\begin{example}\label{1}
Let $I = [0,1]$ be a closed interval, let 
$f_t : (\R,0) \to \R$, $t \in I$, be the polynomial 
defined by $f_t(x) := t^2 - x^2$, 
and let 
$$
\mathcal{F} := \{ f_t \in P(1,1,2;\R ) : t \in I \} . 
$$
Then there is a unique  local topological $\mathcal{RL}_0$ type appearing in 
$\mathcal{F}$, but the local topological $\mathcal{R}_0$ types 
appearing in $\mathcal{F}$ have the cardinal number of the continuum 
since the $\mathcal{R}_0$ equivalence preserves the image.  This example shows, in particular, that Fukuda's finiteness result, Theorem  \ref{fukudafinite}, cannot be generalised to $\mathcal{R}_0.$
\end{example}

Using a similar argument to the proof of Theorem \ref{fukudalofinite0} 
with the semialgebraic versions of Thom's Isotopy Lemmas, 
we can show the following finiteness theorem.

\begin{thm}\label{semialgebraiclo0} 
The number of local semialgebraic $\mathcal{R}_0\mathcal{L}_0$  types appearing in 
$P_0(n,1,d;\R )$ is finite. 
\end{thm}

We have a similar corollary also in the semialgebraic case.

\begin{cor}\label{semialgebraiclo1}
(local semialgebraic $\mathcal{RL}_0$ finiteness theorem)
The number of local semialgebraic $\mathcal{RL}_0$ types appearing in 
$P(n,1,d;\R )$ is finite. 
\end{cor}


\subsection{General setting for finiteness of polynomial functions}\label{generaleq}

Let $(X, \sim_X)$ and $(Y,\sim_Y)$ be two sets endowed with some equivalence 
relations, and let $\pi:X\to Y$ be a surjective mapping. If $a\sim_X b$ implies $\pi(a)\sim_Y\pi(b)$ (or respectively $\pi(a)\sim_Y\pi(b)$ implies  $a\sim_X b$ ) then whenever the quotient space of $X$ is finite then the quotient space of $Y$ is finite (or respectively, whenever  the quotient space of $Y$ is finite then the quotient space of $X$ is finite).
In particular, in the case when $f\sim_X g$ if and only if 
$\tilde f := \pi (f) \sim_Y \tilde g := \pi (g)$, 
 the quotient space of $Y$ is finite if and only if  the quotient 
space of $X$ is finite.
We are going to apply these observations for the relations 
and situations mentioned below.

For $a \in \R^n,$ let $\mu_a$ denote the translation in the direction $\vec{a}$ 
on $\R^n$, and for $c \in \R,$ let $\tau_c$ denote the translation in 
the direction $\vec{c}$ on $\R$. 

\vspace{3mm}

\noindent 
{\bf [1] Local equivalence case.}

{\bf Local topological euivalence of polynomials.}

1) Let $X = P\times \R^n = P(n,1,d;\R)\times \R^n$, 
$Y=P_{0}= \{f\in P| f(0)=0\}$, and 
$\pi(f,a) = \tilde f = f \circ \mu_a - f(a) = \tau_{-f(a)} \circ f \circ \mu_a$, so that $\tilde f(0)=0$. Consider the case where 
$\sim_X = \mathcal{R} \vec{\mathcal{L}}$  is the local right-left equivalence 
with translation (defined below) $\tau \circ f \circ \phi = g$, where $\phi$ is a local 
homeomorphism and $\tau$ is a translation on $\R$, 
and $\sim_Y = \mathcal{R}_0$ is the local topological right equivalence.

 Given $(f,a), \ (g,b) \in P(n,1,d;\R) \times \R^n$ we say that they are locally topologically 
 $\mathcal{R} \vec{\mathcal{L}}$ equivalent if 
 there exists a local homeomorphism $\phi : (\R^n,b) \to (\R^n,a)$ 
such that $\tau_{g(b)-f(a)} \circ f \circ \phi = g$. 
Therefore we have 
$$
(\tau_{-f(a)} \circ f \circ \mu_a) \circ (\mu_{-a} \circ \phi \circ 
\mu_b ) = \tau_{-g(b)} \circ g \circ \mu_b.
$$
Since $\mu_{-a} \circ \phi \circ \mu_b$ is a local homeomorphism : 
$(\R^n,0) \to (\R^n,0)$,  $\tilde f$ and $\tilde g$ are locally topologically 
$\mathcal{R}_0$ equivalent.

Conversely we assume that $\tilde f$ and $\tilde g$ are locally topologically 
$\mathcal{R}_0$ equivalent.
Then there exists a local homeomorphism $\phi : (\R^n,0) \to (\R^n,0)$ 
such that $(\tau_{-f(a)} \circ f \circ \mu_a) \circ \phi 
= \tau_{-g(b)} \circ g \circ \mu_b$.
Therefore we have $\tau_{g(b)-f(a)} \circ f \circ 
(\mu_a \circ \phi \circ \mu_{-b}) = g$. 
Since $\mu_a \circ \phi \circ \mu_{-b}$ is a local homeomorphism : 
$(\R^n,b) \to (\R^n,a)$, i.e.  $(f,a)$ and $(g,b)$ 
are $\mathcal{R} \vec{\mathcal{L}}$ equivalent in our sense. 

We have shown that $(f,a)\sim_X (g,b)$ if and only if $\pi(f,a)=\tilde f\sim_Y \tilde g=\pi(g,b)$. 

\vspace{3mm}

2) Let $X = P$, $Y=P_0 = P_0(n,1,d;\R)$ and 
$\pi(f)=\tilde f = f - f(0) = \tau_{-f(0)} \circ f$, i.e. the case when both $a=b=0$.
Consider the case when $\sim_X = \mathcal{R} \vec{\mathcal{L}}$ is the local right-left equivalence 
with translation as above, and $\sim_Y = \mathcal{R}_0$ is the local topological 
right equivalence.

Using a similar argument to case 1), we can  easily show in this case 
that $f\sim_X g$ if and only if $\tilde f\sim_Y \tilde g$. 

\vspace{3mm}

3) Let $X = P$, $Y=P_0= P_0(n,1,d;\R)$ and 
$\pi(f)=\tilde f=f-f(0) = \tau_{-f(0)} \circ f$.
Both equivalences are the local topological right-left equivalence $\mathcal{RL}$.

Using a similar argument to the global topological equivalence case below, 
we can show also in this local case that $f\sim_X g$ if and only if 
$\tilde f\sim_Y \tilde g$. 

\vspace{3mm}

By the above observations and discussions, we have the following.

\begin{note}\label{note1}
i) Consider case 1); By Theorem \ref{localF}, the number of local topological $\mathcal{R}_0$ 
orbits in $P_0$ is finite.   
Therefore the number of  local topological 
 $\mathcal{R} \vec{\mathcal{L}}$ orbits in $P\times \R^n$ is finite.

ii) Consider case 2);
By Theorem \ref{localF}, the number of local topological $\mathcal{R}_0$ 
orbits in $P_0$ is finite.   
Therefore the number of local topological 
 $\mathcal{R} \vec{\mathcal{L}}$ orbits in $P$ is finite. 

iii) Consider case 3);
By Theorem \ref{fukudalofinite0}, 
the number of local topological $\mathcal{RL}$ orbits in $P_0$ is finite.
Therefore the number of local topological 
 $\mathcal{RL}$ orbits at $0\in \R^n$ in $P$  is finite. 
This is also a consequence of the previous fact ii).
\end{note}

\vspace{2mm}

{\bf Local Semialgebraic equivalence of polynomials.}

In the semialgebraic equivalence case, we can show a similar note 
to Note \ref{note1}, using Theorems \ref{localBS} and 
\ref{semialgebraiclo0}.

\vspace{2mm}

\noindent 
{\bf [2] Global equivalence case.}

{\bf  Topological equivalence of polynomial functions.}

Let $X = P = P(n,1,d;\R)$, $Y=P_0=\{f\in P \ | \ f(0)=0\}$ and 
$\pi(f)=\tilde f=f-f(0) = \tau_{-f(0)} \circ f$.
Both equivalences are the topological righ-left equivalences $\mathcal{RL}$. 

Let $f, g \in P(n,1,d;\R)$.
Assume that $f$ and $g$ are topologically $\mathcal{RL}$ equivalent. 
Then there exist homeomorphisms $\phi : \R^n \to \R^n$ 
and $\psi : \R \to \R$ such that $\psi \circ f \circ \phi = g$. 
Therefore we have 
$$
(\tau_{-g(0)} \circ \psi \circ \tau_{f(0)}) \circ 
(\tau_{-f(0)} \circ f) \circ \phi = \tau_{-g(0)} \circ g. 
$$
It follows that $\tilde f$ and $\tilde g$ are topologically $\mathcal{RL}$ 
equivalent.

Conversely we assume that $\tilde f$ and $\tilde g$ are topologically 
$\mathcal{RL}$ equivalent.
Then there exist homeomorphisms $\phi : \R^n \to \R^n$ and 
$\psi : \R \to \R$ such that 
$\psi \circ (\tau_{-f(0)} \circ f) \circ \phi = \tau_{-g(0)} \circ g$. 
Therefore we have $(\tau_{g(0)} \circ \psi \circ \tau_{-f(0)}) \circ f 
\circ \phi = g$.
Hence $f$ and $g$ are topologically $\mathcal{RL}$ equivalent.

It follows that $f\sim_X g$ if and only if $\tilde f\sim_Y \tilde g$. 

We have the following note in the global topological equivalence case.

\begin{note}\label{note2}
By Theorem \ref{fukudafinite}, the number of topological $\mathcal{RL}$ 
orbits in $P$ is finite.   
Therefore the numbers of topological $\mathcal{RL}$ orbits in $P$ and $P_0$ 
are simultaneously finite. 
\end{note}

\vspace{2mm}

{\bf Semialgebraic equivalence of polynomial functions.}

In this case we can show a similar note to Note \ref{note2}, 
using Theorem \ref{benedettishiota}.

The above considerations show that one can concentrate on proving  finiteness results 
for $P_0=P_0(n,1,d;\R)$, for instance  all  local cases are a consequence of 
the local topological right equivalence on $P_0$. The examples \ref{1} and \ref{2} show that we do not have local or global finiteness for the $\mathcal{R} \vec{\mathcal{L}}$ equivalence in $P(n,1,d;\R^n)$.


\subsection{Finiteness for $\mathcal{R}$-equivalence of polynomial function 
germs}\label{localpoly2}

The finiteness results may hold for a stronger equivalence than topological 
$\mathcal{RL}$ equivalence.
In fact, after proving Theorem \ref{fukudalofinite0}, 
T. Fukuda 
in his lectures \cite{fukuda3} pointed out that the following 
local finiteness theorem holds. 

\vspace{3mm}

\noindent
{\bf Theorem \ref{localF}.} 
(Fukuda's local $\mathcal{R}_0$ finiteness theorem) 
{\em The number of local topological $\mathcal{R}_0$ types 
appearing in $P_0(n,1,d;\R )$ is finite.
In addition, each local topological $\mathcal{R}_0$ equivalence class in 
$P_0(n,1,d;\R )$ is a semialgebraic subset of $P_0(n,1,d;\R )$.} 

\vspace{3mm}

Apart from mentioning the result in \cite{fukuda3},  it seems that Fukuda has not written up  its statement and his idea of proof (we could not find the result in the published literature).
Therefore we decided to give a proof for Fukuda's  local $\mathcal{R}$  finiteness result in this paper. 
The latter semialgebraiceness statement follows from the proof of the first one.

Before we start the proof, we prepare for some notion.
We say that two polynomials $f, g : (\R^n,0) \to (\R,0)$ 
are {\em locally topologically $\mathcal{R-IL}$ equivalent}, if there exist 
local homeomorphisms $\sigma : (\R^n,0) \to (\R^n,0)$ and 
$\tau : (Im(f),0) \to (Im(g),0)$ such that 
$\tau \circ f = g \circ \sigma$, where $Im(f)$ denotes the image of $f$. 

\begin{rem}\label{remark40}
Suppose that $f, g \in P_0(n,1,d;\R )$ are locally topologically $\mathcal{R-IL}$ 
equivalent with local homeomorphisms $\sigma : (\R^n,0) \to (\R^n,0)$ and 
$\tau : (Im(f),0) \to (Im(g),0)$.
If $\tau$ is the identity mapping on $Im(f)$ around $0 \in \R$, 
then $f$ and $g$ are locally topologically $\mathcal{R}_0$ equivalent. 
\end{rem}

\begin{proof}[Proof of Theorem \ref{localF}]

Using the Curve Selection Lemma, we can easily see the following.

\begin{prope}\label{singularpts}
Let $f : (\R^n,0) \to (\R,0)$ be a $C^{\omega}$ function germ. 
Suppose that $0 \in \R^n$ is a singular point of $f$. 
Then the singular points set $S(f)$ of $f$ is contained in $f^{-1}(0)$ 
as germs at $0 \in \R^n$.
\end{prope}

As mentioned before, $P_0(n,1,d;\R )$ is naturally identified with 
the Euclidean space $\R^{L}$. 
Let us express the correspondence by $\R^L \ni r \leftrightarrow 
f_r \in P_0(n,1,d;\R )$. 
Using Theorem \ref{tarski-seidenberg}, we can show that 
$\{ (f_r(S(f_r )), r) \in \R \times \R^L \}$ 
is a semialgebraic subset of $\R \times \R^L$. 
Therefore we have the following observation from Theorem \ref{lojasiewicz}. 

\begin{obs}\label{criticalvalue}
Let $Q \subset \R^L$ be a Nash manifold of dimension $\ell$. 
Then 
$$
\overline{\{ (f_r(S(f_r )), r) \in \R \times Q \ | \ f_r \in P_0(n,1,d;\R ) \} 
\setminus \{ 0 \} \times Q } \cap \{ 0 \} \times Q 
$$
is a semialgebraic subset of $\{ 0 \} \times Q$ of dimension less than $\ell$, 
where the closure is taken in $\R \times Q$.
\end{obs}

Let $F : (\R^n \times \R^L, \{ 0 \} \times \R^L) \to (\R, \{ 0 \})$ be 
a polynomial function defined by $F(x,r) := f_r(x)$. 
Then we define the map
$$
\Gamma : (\R^n \times \R^L, \{ 0 \} \times \R^L) \to 
(\R \times \R^L, \{ 0 \} \times \R^L)
$$ 
by $\Gamma := (F, id_{\R^L})$, namely $\Gamma (x,r) = (f_r(x),r)$. 
Let $q : \R \times \R^L \to \R^L$ be the canonical projection. 
For $Q \subset \R^L$, we set $F_Q := F |_{\R^n \times Q}$ .

Let us consider the situation when Thom's 2nd Isotopy Lemma 
(Theorem \ref{2ndisotopy}) 
is applicable to the above map-germs $\Gamma$ and $q$ after 
stratifying $\R^n \times \R^L$ near $\{ 0\} \times \R^L$ 
and $\R \times \R^L$ near $\{ 0\} \times \R^L$ and 
taking a finite subdivision of $\R^L$ into connected Nash manifolds
(in order to see Theorem \ref{fukudalofinite0}, 
that is the finiteness for local topological 
$\mathcal{RL}$ types appearing in $P_0(n,1,d;\R )$). 
Note that for $f \in P_0(n,1,d;\R )$, Im($f$) $= \{ 0 \}$ as germs at 
$0 \in \R$ if and only if $f$ is a $0$-mapping. 
Set $\{ 0$-map$\} \subset \R^L$ as $Q_0$. 
Therefore we consider the situation when $\Gamma$ and $q$ are stratified 
in the following way.
Namely, there exists a finite partition of $\R^L \setminus Q_0$ into connected 
Nash manifolds $\R^L \setminus Q_0 = \cup_{i = 1}^u Q_i$ such that for each 
$i = 1, \cdots , u$, there are closed semialgebraic neighbourhoods $A_i$ 
of $\{ 0 \} \times Q_i$ in $\R^n \times Q_i$ and $C_i$ of $\{ 0 \} \times Q_i$ 
in $\R \times Q_i$ with the following properties:

\vspace{2mm} 

Set $B_i := C_i \cap \text{Im}(\Gamma |_{A_i})$.

(1)(Case I) If $f \in Q_i$ takes both positive and negative values as a 
function germ at $0 \in \R^n$, then we have 
$C_i \subset \text{Im}(\Gamma |_{A_i})$.

(Case II) If $f \in Q_i$ does not take negative values as a function 
germ at $0 \in \R^n$, then we have 
$C_i \cap \{ y \in \R \ | \ y \ge 0 \} \times Q_i \subset 
\text{Im}(\Gamma |_{A_i})$. 

(Case III) If $f \in Q_i$ does not take positive values as a function 
germ at $0 \in \R^n$, then we have 
$C_i \cap \{ y \in \R \ | \ y \le 0 \} \times Q_i \subset 
\text{Im}(\Gamma |_{A_i})$.

(2) $\Gamma |_{A_i} : A_i \to B_i$ and $q |_{B_i} : B_i \to Q_i$ are proper.

(3) $A_i$ and $B_i$ admit finite Whitney stratifications of Nash class 
$\mathcal{S}(A_i)$ and $\mathcal{S}(B_i)$, respectively, 
such that $\Gamma : (A_i,\mathcal{S}(A_i)) \to (B_i,\mathcal{S}(B_i))$ is 
a Thom regular stratified mapping and 
$q : (B_i,\mathcal{S}(B_i)) \to (Q_i,\{ Q_i \} )$ is a stratified mapping. 

\vspace{2mm}

\noindent
By Theorem \ref{lojasiewicz} and Lemma \ref{stratifiedmapping}, taking 
a finite subdivision of $Q_i$ and substratifying $\mathcal{S}(A_i)$ and 
$\mathcal{S}(B_i)$ if necessary, we may assume that 

\vspace{2mm}

(4) The stratification $\mathcal{S}(A_i)$ is compatible with 
$\{ 0 \} \times Q_i$, $\{ (S(f_r), r) \ | \ r \in Q_i \}$ and 
$F_{Q_i}^{-1}(0)$, 
the stratification $\mathcal{S}(B_i)$ is compatible with 
$\{ 0 \} \times Q_i$ and $\{ f_r(S(f_r)), r) \ | \ r \in Q_i \}$, and 
$Q_i$ is a Nash open simplex. 

\vspace{2mm}

By Observation \ref{criticalvalue} and Theorem \ref{stratifiedmapping}, 
after taking a finite subdivision of $Q_i$ into Nash open simplices 
if necessary, we can assume that  
\begin{equation}\label{approach}
\overline{\{ (f_r(S(f_r )), r) \in \R \times Q_i \ | \ 
f_r \in P_0(n,1,d;\R ) \} \setminus \{ 0\} \times Q_i } \cap 
\{ 0 \} \times Q_i = \emptyset , 
\end{equation}
$1 \le i \le u$. 
Taking a further subdivision of $Q_i$ into finite Nash open simplices 
if necessary, we may assume that if $U \cap \{ 0 \} \times Q_i \ne \emptyset$ 
for $U \in \mathcal{S}(B_i)$, then $U = \{ 0 \} \times Q_i$, and that
$$
\{ (f_r(S(f_r )), r) \in \R \times Q_i \ | \ 
f_r \in P_0(n,1,d;\R ) \} = U
$$
as set-germs at $\{ 0 \} \times Q_i$, if it is not empty.
In addition, if $\overline{V} \cap \{ 0 \} \times Q_i \neq \emptyset$ for 
$V \in \mathcal{S}(B_i)$, then $\overline{V} \supset \{ 0 \} \times Q_i = U$. 
Let us denote by $V_1$ and $V_2$ such two adjacent strata of $\mathcal{S}(B_i)$ 
to $U$ in Case I, and by $V$ such only one adjacent stratum of $\mathcal{S}(B_i)$ 
to $U$ in Cases II and III.
We set $V_0 := U$.

In order to see Theorem \ref{localF} holds, it suffices to show that for 
$r_1, r_2 \in Q_i$, $i = 1, \cdots , u$, $f_{r_1}$ and $f_{r_2}$ 
are topologically $\mathcal{R}$ equivalent.
Since $Q_i$ is a Nash open simplex, there is a Nash diffeomorphism 
$\tau : Q_i \to \Delta_i $ between $Q_i$ and an open simplex $\Delta_i$ 
in some Euclidean space. 
Let $s_1 = \tau (r_1)$ and $s_2 = \tau (r_2)$, and let 
$h : I \to \Delta_i$ be the homotopy defined by $h(t) = (1 - t) s_1 + t s_2$ 
for $t \in I = (-\eta ,1 + \eta )$. 
Here $\eta$ is a sufficiently small 
positive number such that if we set $J := \tau^{-1} \circ h(I)$, 
$\overline{J} \subset Q_i$ where the  closure is taken in $\R^L$.
Then $J$ is a Nash curve without end points.
Therefore $\R^n \times J$ and $\R \times J$ are Nash manifolds of dimension 
$n + 1$ and $2$, respectively.

We consider the restriction of $q \circ \Gamma$ over $J \subset \R^L$: 
\begin{equation}\label{ristriction}
\begin{CD}
(A_i \cap \R^n \times J, \{ 0\} \times J) 
@>\text{$\tilde{\Gamma}$}>> (B_i \cap \R \times J, \{ 0\} \times J)
@>\text{$\tilde{q}$}>> J 
\end{CD}
\end{equation}
We define the vector field $v$ on $J$ by 
$v := D(\tau^{-1} \circ h)(\frac{\partial}{\partial t})$. 
Note that $v$ is tangent to $J \subset \R^L$. 

When we show the local topological triviality of 
$\{ f_{\gamma} \}_{\gamma \in J}$ using Thom's 2nd Isotopy Lemma, 
we lift the vector field $v$ on $J$ to a neighbourhood $D_i$ of 
$\{ 0\} \times J$ in $B_i \cap \R \times J$ by $d\tilde{q}$ and then to 
a neighbourhood $E_i$ of $\{ 0\} \times J$ in $A_i \cap \R^n \times J$ 
by $d(\tilde{q} \circ \tilde{\Gamma})$ so that the first lifted 
vector field is controlled by a controlled tube system 
$\{T^{\prime}_j = (|T^{\prime}_j |, \pi^{\prime}_j, \rho^{\prime}_j)\}$ 
for $\{ V_j \in \mathcal{S}(B_i \cap \R \times J)\}$ 
and the second lifted vector field is controlled by a controlled 
tube system $\{T_k = (|T_k |, \pi_k, \rho_k)\}$ 
for $\{ W_k \in \mathcal{S}(A_i \cap \R^n \times J)\}$ controlled 
over $\{T^{\prime}_j\}$.

As the neighbourhood $D_i$ of $\{ 0\} \times J$ in $B_i \cap \R \times J$, 
we take $[-\epsilon_0 , \epsilon_0] \times J$, $[0, \epsilon_0] \times J$ 
and $[-\epsilon_0, 0] \times J$ ($\epsilon_0 > 0$) in Cases I, II and III, 
respectively.
Here we take $\epsilon_0 > 0$ sufficiently small so that
$[-\epsilon_0 , \epsilon_0] \times J \subset V_1 \cup V_0 \cup V_2$, 
$[0, \epsilon_0] \times J \subset V \cup V_0$ and 
$[-\epsilon_0, 0] \times J \subset V \cup V_0$ in Cases I, II and III, 
respectively.
Let $\hat{V}_0 := \{ 0\} \times J$, $\hat{V_1} := (0,\epsilon_0) \times J$, 
$\hat{V}_2 := (-\epsilon_0,0)$, $\hat{V}_3 := \{ \epsilon_0 \} \times J$ 
and $\hat{V}_4 := \{ -\epsilon_0\} \times J$. 
We consider canonical tubular neighborhood systems (in 
$D_i \subset \R \times J$) for $\{ \hat{V_j}\}$ as follows: 
$\hat{T}^{\prime}_0 = (|\hat{T}^{\prime}_0|, \hat{\pi}^{\prime}_0, 
\hat{\rho}^{\prime}_0)$ where $|\hat{T}^{\prime}_0| := 
\{ (y,s) \ | \ \hat{\rho}^{\prime}_0 (y,s) \le \epsilon_1^2 \}$, 
$0 < \epsilon_1 \le \epsilon_0$,
$\hat{\pi}^{\prime}_0 (y,s) = (0,s)$ and $\hat{\rho}^{\prime}_0(y,s) = y^2$; 
$\hat{T}^{\prime}_3 = (|\hat{T}^{\prime}_3|, \hat{\pi}^{\prime}_3, 
\hat{\rho}^{\prime}_3)$ where $|\hat{T}^{\prime}_3| := 
\{ (y,s) \ | \ \hat{\rho}^{\prime}_3 (y,s) < \epsilon_3^2 \}$, 
$0 < \epsilon_3 \le \epsilon_0$, 
$\hat{\pi}^{\prime}_3 (y,s) = (\epsilon_0 ,s)$ and  
$\hat{\rho}^{\prime}_3(y,s) = (y - \epsilon_0)^2$; 
$\hat{T}^{\prime}_4 = (|\hat{T}^{\prime}_4|, \hat{\pi}^{\prime}_4, 
\hat{\rho}^{\prime}_4)$ where $|\hat{T}^{\prime}_4| := 
\{ (y,s) \ | \hat{\rho}^{\prime}_4 (y,s) < \epsilon_4^2 \}$, 
$0 < \epsilon_4 \le \epsilon_0$, 
$\hat{\pi}^{\prime}_4 (y,s) = (-\epsilon_0 ,s)$ and 
$\hat{\rho}^{\prime}_4(y,s) = (y + \epsilon_0)^2$; 
$\hat{T}^{\prime}_1 = (|\hat{T}^{\prime}_1|, \hat{\pi}^{\prime}_1, 
\hat{\rho}^{\prime}_1)$ where $|\hat{T}^{\prime}_1| = \hat{V}_1$, 
$\hat{\pi}^{\prime}_1 = id_{\hat{V}_1}$ and $\hat{\rho}^{\prime}_1 = 0$;
$\hat{T}^{\prime}_2 = (|\hat{T}^{\prime}_2|, \hat{\pi}^{\prime}_2, 
\hat{\rho}^{\prime}_2)$ where $|\hat{T}^{\prime}_2| = \hat{V}_2$, 
$\hat{\pi}^{\prime}_2 = id_{\hat{V}_2}$ and $\hat{\rho}^{\prime}_2 = 0$. 
Then we define the first lifted vector field $w$ of $v$ on $D_i$ 
by $w(y,s) := (0,v(s))$.
This is a vector field controlled by the controlled tube system 
of Nash class  
$\{ \hat{T}^{\prime}_j \ | \ j = 0, 1, 2, 3, 4\}$,
$\{ \hat{T}^{\prime}_j \ | \ j = 0, 1, 3\}$ and 
$\{ \hat{T}^{\prime}_j \ | \ j = 0, 2, 4\}$ for 
$\{ \hat{V}_j \in \mathcal{S}([-\epsilon_0 , \epsilon_0] \times J)\}$,
$\{ \hat{V}_j \in \mathcal{S}([0 , \epsilon_0] \times J)\}$ and 
$\{ \hat{V}_j \in \mathcal{S}([-\epsilon_0 , 0] \times J)\}$ 
in Cases I, II and III, respectively.
Actually, $d\hat{\rho}^{\prime}_j (w) = 0$, $j = 0, 1, 2, 3, 4$. 
In particular, $w$ is tangent to 
$\{ (y,s) \in D_i \ | \ y = \text{constant} \} $ in 
Cases I, II and III. 

Let $W_0 := \{ 0\} \times J \subset \R^n \times J$, and let 
$W_1, \cdots , W_d \in \mathcal{S}(A_i)$ be all the adjacent strata 
to $W_0$.
As the neighbourhood $E_i$ of $\{ 0\} \times J$ in $A_i \cap \R^n \times J$, 
we take $\tilde{\Gamma}^{-1}(D_i) \cap A_i$. 
We may assume that $E_i \subset W_0 \cup W_1 \cup \cdots \cup W_d.$ 
Let $\mathring{D}_i$ be $(-\epsilon_0 , \epsilon_0) \times J$, 
$[0, \epsilon_0) \times J$ and $(-\epsilon_0, 0] \times J$ in Cases 
I, II and III, respectively, and let 
$\mathring{E}_i := \tilde{\Gamma}^{-1}(\mathring{D}_i) \cap A_i$. 
We further let $\breve{D}_i := \{ \pm \epsilon_0 \} \times J$, 
$\{ \epsilon_0\} \times J$ and $\{ - \epsilon_0\} \times J$ in Cases 
I, II and III, respectively, and let 
$\breve{E}_i := \tilde{\Gamma}^{-1}(\breve{D}_i) \cap A_i$. 
Set $\mathcal{S}(\mathring{E}_i ) := \{ W \cap \mathring{E}_i \ | \ 
W \in \mathcal{S}(A_i)\}$ and 
$\mathcal{S}(\breve{E}_i ) := \{ \text{connected componets of} \ 
\breve{E}_i \}$. 
Define $\mathcal{S}(E_i) := \mathcal{S}(\mathring{E}_i ) \cup 
\mathcal{S}(\breve{E}_i )$. 

Now we lift the vector field $w$ on $E_i$ so that the lifted vector field 
$\beta$ is controlled by a controlled tube system of Nash class 
$\{\hat{T}_k = (|\hat{T}_k |, \hat{\pi}_k, \hat{\rho}_k)\}$ 
for $\{\hat{W}_k \in \mathcal{S}(E_i)\}$ controlled 
over $\{\hat{T}^{\prime}_j\}$. 
The flows of $\beta$ and $w$ yield the topological triviality of 
$\{ f_{\gamma} \}_{\gamma \in J}$. 
By construction, the flow of $w$ yields the triviality keeping 
``$y = $constant". 
Since $\{ 0\} \times J \in \mathcal{S}(E_i)$,
$\beta$ is tangent to $\{ 0\} \times J$. 
Therefore $f_{r_1}$ and $f_{r_2}$ are locally topologically 
$\mathcal{R-IL}_0$ equivalent with local homeomorphisms 
$\sigma : (\R^n,0) \to (\R^n,0)$ and $\tau : (Im(f),0) \to (Im(g),0)$ 
such that $\tau$ is the identity mapping on $Im(f)$.
It follows from Remark \ref{remark40} that $f_{r_1}$ and $f_{r_2}$ are 
locally topologically $\mathcal{R}_0$ equivalent. 
\end{proof}

\begin{rem}\label{remark41}
Let us regard $P_0(n,1,d;\R )$ as a family of polynomials, indexed by $r\in \R^L$.
From the proof of Theorem \ref{fukudalofinite0} using Thom's 2nd Isotopy Lemma, 
we can see that a finiteness result for the local $\mathcal{RL}$ triviality 
of the family $P_0(n,1,d;\R )$ holds. 
On the other hand, we cannot see from the proof of the above theorem 
whether a finiteness result for the local $\mathcal{R}$ triviality of the 
family $P_0(n,1,d;\R )$ holds or not. 

In the proof of Theorem \ref{localF}  we could assume that 
$$
\overline{\{ (f_r(S(f_r )), r) \in \R \times Q_i \ | \ 
f_r \in P_0(n,1,d;\R ) \} \setminus \{ 0\} \times Q_i } \cap 
\{ 0 \} \times Q_i = \emptyset .
$$
Therefore there exists a positive constant $\epsilon_0 > 0$ such that 
in Cases I, II and III, $[-\epsilon_0 , \epsilon_0] \times J$, $[0, \epsilon_0] \times J$ 
and $[-\epsilon_0, 0] \times J$, respectively, is contained in
$$
B_i \cap \R \times J \setminus 
\overline{\{ (f_s(S(f_s )), s) \in \R \times J \ | \ 
f_s \in P_0(n,1,d;\R ) \} \setminus \{ 0\} \times J } , 
$$
where the closure is taken in $\R \times J$.
Remember that $\overline{J} \subset Q_i$, where the closure is taken in $\R^L$. 
In general, even if we take a positive constant $\epsilon_0 > 0$ arbitrarily small, 
in Cases I, II and III, $[-\epsilon_0 , \epsilon_0] \times Q_i$, $[0, \epsilon_0] \times Q_i$ 
and $[-\epsilon_0, 0] \times Q_i$, respectively, may not 
be contained in 
$$
B_i \setminus 
\overline{\{ (f_r(S(f_r )), r) \in \R \times Q_i \ | \ 
f_r \in P_0(n,1,d;\R ) \} \setminus \{ 0\} \times Q_i }. 
$$
In addition, taking a finite subdivision of $Q_i$ into Nash open simplices, 
we cannot avoid this situation. 
Therefore we cannot always see a local $\mathcal{R}$ triviality over $Q_i$, 
using a similar argument to the case of local $\mathcal{R}$ equivalence.
\end{rem}

\begin{example}\label{2}
Let $I = [0,1]$ be a closed interval, let 
$g_t : (\R,0) \to (\R,0)$, $t \in I$, be a polynomial defined 
by $g_t(x) := t^2 - (x - t)^2$, 
and let 
$$
\mathcal{G} := \{ g_t \in P_0(1,1,2;\R ) : t \in I \} . 
$$
Then the number of local topological $\mathcal{R}$ types appearing in 
$\mathcal{G}$ is 2, but the global topological $\mathcal{R}$ types appearing in 
$\mathcal{G}$ as global polynomial functions have the cardinal number of 
the continuum.  Note this family is $\mathcal{R}$ equivalent in $P(n,1,2;\R)$ with the family in the example \ref{1}.
\end{example}

By Fact 2.1 in M. Shiota \cite{shiota3} (Theorem II.7.1 in M. Shiota 
\cite{shiota2}), we have the following fact in the case of semialgebraic 
equivalence.

\begin{thm}\label{shiotabirk}
Let $X \subset Y \subset \R^n$ be semialgebraic sets, and let 
$f, g : Y \to \R$ be semialgebraic functions with 
$f^{-1}(0) = g^{-1}(0) = X$. 
If the germs of $f$ and $g$ at $X$ are semialgebraically 
$\mathcal{RL}$ equivalent, then the germs of $f$ and $g$ are 
semialgebraically $\mathcal{R}$ equivalent or the germs of $f$ 
and $-g$ are so.
Here we can choose the semialgebraic homeomorphisms of 
equivalence to be the identity mapping on $X$.
\end{thm}

As a corollary of this theorem, we have the following result.

\begin{cor}\label{RvsRL}
Let $f, g \in P_0(n,1,d;\R )$.
If $f$ and $g$ are locally semialgebraically $\mathcal{R}_0\mathcal{L}_0$ equivalent, then
$f$ and $g$ are locally semialgebraically $\mathcal{R}_0$ equivalent or 
$f$ and $-g$ are locally semialgebraically $\mathcal{R}_0$ equivalent.
\end{cor}

\begin{proof}
Suppose that $f, g \in P_0(n,1,d;\R )$ are locally semialgebraically $\mathcal{R}_0\mathcal{L}_0$
equivalent. 
Therefore there exist semialgebraic homeomorphisms 
$\sigma : U \to V$ with $\sigma (0) = 0$ and $\tau : P \to Q$ 
with $\tau (0) = 0$ such that $\tau \circ f \circ \sigma = g$, where 
$U$ and $V$ (respectively $P$ and $Q$) are open semialgebraic 
neighbourhoods of $0$ in $\R^n$ (respectively $0$ in $\R$). 
Set $\tilde{f} := f \circ \sigma$ and $\tilde{g} := g|_U$. 
Then $\tilde{f}, \tilde{g} : U \to \R$ are semialgebraic functions, 
and $\tilde{f}^{-1}(0) = \tilde{g}^{-1}(0)$ denoted by $X$ after this. 
In addition, $\tilde{f}$ and $\tilde{g}$ are locally semialgebraically 
$\mathcal{R}_0\mathcal{L}_0$ equivalent. 
Here the source homeomorphism of the local semialgebraic $\mathcal{R}_0\mathcal{L}_0$
equivalence is the identity mapping on $U$. 
Therefore $\tilde{f}$ and $\tilde{g}$ are locally semialgebraically 
$\mathcal{R}_0\mathcal{L}_0$ equivalent as function germs at $X \subset U$. 
By Theorem \ref{shiotabirk}, $\tilde{f}$ and $\tilde{g}$ are 
semialgebraically $\mathcal{R}_0$ equivalent or $\tilde{f}$ and 
$-\tilde{g}$ are semialgebraically $\mathcal{R}_0$ equivalent 
as function germs at $X$. 
Since we can choose the semialgebraic homeomorphisms of equivalence 
to be the identity mapping on $X$, it follows that $f$ and $g$ are 
semialgebraically $\mathcal{R}_0$ equivalent or $f$ and $-g$ are 
semialgebraically $\mathcal{R}_0$ equivalent as function germs 
at $0 \in \R^n$.
\end{proof}

Combining this corollary with Theorem \ref{semialgebraiclo0}, we have
the following finiteness theorem in the semialgebraic case.

\vspace{3mm}

\noindent
{\bf Theorem \ref{localBS}.}  
{\em The number of local semialgebraic $\mathcal{R}_0$ types 
appearing in $P_0(n,1,d;\R )$ is finite.} 

\vspace{3mm}

\noindent
We may call the above result {\em Shiota's local semialgebraic $\mathcal{R}_0$ 
finiteness theorem}.
 
At the end of this subsection, we make an interesting remark on $\mathcal{R}_0$ 
equivalence of polynomial function germs, which was shown by M. Shiota.

\begin{rem}\label{remark42}
Shiota proved that there exist two homogeneous polynomials 
$f, g : (\R^7,0) \to (\R,0) $ of the same degree with isolated singularities 
which are locally topologically $\mathcal{R}_0$ equivalent but not locally semialgebraically 
$\mathcal{R}_0$ equivalent.
(See Example II.7.9 in \cite{shiota2} or Fact 2.3 in \cite{shiota3}).
\end{rem}


\bigskip
\section{A generalisation.}
\label{generalisation}

In this section we treat a generalisation of the Fukuda-Benedetti-Shiota 
theorem. 
We can read Fukuda's finiteness theorem (respectively Benedetti-Shiota's one) 
from its proof as a finiteness result on topological $\mathcal{RL}$ 
triviality (respectively semialgebraic $\mathcal{RL}$ one) for a polynomial family 
of polynomial functions as follows:

\begin{thm}\label{fukudabenedettishiota} 
Given a polynomial family of polynomial functions from $\R^n$ 
to $\R$, there exists a finite subdivision of the parameter space 
into Nash open simplices such that over each Nash open simplex
the family of polynomial functions is topologically $\mathcal{RL}$ trivial 
(respectively semialgebraically $\mathcal{RL}$ trivial).
\end{thm}

We define the notions of {\em topological $\mathcal{RL}$ equivalence} 
and {\em semialgebraic $\mathcal{RL}$ equivalence} for Nash functions defined 
on a Nash manifold $M$ in a similar way to the above polynomial case 
with $M = \R^n$.
Then we can generalise Theorem \ref{fukudabenedettishiota} to the 
following form.

\vspace{3mm}

\noindent {\bf Theorem \ref{finitenessNF}.}
{\em Let $\{ f_t : M \to \R \ | \ t \in J \}$ be a Nash family of Nash 
functions defined on a Nash manifold $M$ with a semialgebraic 
parameter space $J$.
Namely, the function $F : M \times J \to \R$ defind by 
$F(x,t) := f_t(x)$ is a Nash function. 
Then there exists a finite subdivision of $J$ into Nash open simplices 
$J = Q_1 \cup \cdots \cup Q_u$ such that 
$$
\{ f_t : M \to \R \ | \ t \in Q_i \} 
$$
is semialgebraically $\mathcal{RL}$ trivial over each $Q_i$.} 

\vspace{3mm}

\begin{rem}\label{remark11}
J. Bochnak, M. Coste and M.-F. Roy  showed  in \cite{bochnakcosteroy}  
a finiteness theorem for semialgebraic triviality of a semialgebraic family 
of semialgebraic functions from $\R^n$ to $\R$.
\end{rem}

\begin{proof}[Proof of Theorem \ref{finitenessNF}]
The Fukuda-Benedetti-Shiota theorem is an algebraic result.
When we show the global triviality using the 2nd Thom's Isotopy Lemma
(Theorem \ref{2ndisotopy}) or its semialgebraic version (Theorem 
\ref{s2ndisotopy}), we can consider the projectivisation in order to get 
the properness of mappings. 
On the other hand, Theorem \ref{finitenessNF} is a generalisation of the
Fukuda-Benedetti-Shiota theorem to the Nash case.
We first reduce the Nash case to the algebraic one in the proof of 
Theorem \ref{finitenessNF}.
This reduction process follows from similar arguments to those
discussed in \cite{koike3, koikeshiota}.

Let us consider the composite of the mapping
$(F,id_J) : M \times J \to \R \times J$ defined by
$(F,id_J)(x,t) := (f_t(x),t)$ and the canonical projection
$q : \R \times J \to J$. 
Here we recall a useful result by M. Shiota.

\begin{thm}\label{algebraicstructure}(M. Shiota \cite{shiota1})
A Nash manifold is Nash diffeomorphic to a nonsingular,
affine algebraic variety.
\end{thm}

By this theorem, we can replace the Nash manifold $M$ with
a nonsigular, affine algebraic variety $W$.
On the other hand, by Theorem \ref{lojasiewicz}, $J$ is a
finite union of Nash open simplices, and a Nash open simplex
is Nash diffeomorphic to some Euclidean space.
Therefore, in order to show our finiteness result, we may assume 
from the beginning that $J$ is an Euclidean space $\R^d$. 
Then the above composed mapping becomes the following: 
\begin{equation}\label{algebraic} 
\begin{CD}
W \times \R^d
@>\text{$\Delta$}>> \R \times \R^d
@>\text{$q$}>> \R^d.
\end{CD}
\end{equation}
Note that $\Delta : W \times \R^d \to \R \times \R^d$ is a 
Nash family of Nash mappings between nonsingular, affine
algebraic varieties.

We next apply the Artin-Mazur theorem (M. Artin and B. Mazur
\cite{artinmazur}, M. Coste, J.M. Ruiz and M. Shiota
\cite{costeruizshiota}, M. Shiota \cite{shiota1}) for a Nash family
of Nash mappings, modified to our situation.
For an algebraic variety $X \subset \R^p$, we denote by $Reg(X)$
the smooth part of $X$. 

\begin{thm}\label{artinmazur} (Artin-Mazur Theorem).
Let $W \subset \R^m$ be the above nonsingular algebraic variety,
and let $\Delta : W \times \R^d \to \R \times \R^d$ be the above
Nash mapping.
Then there exists a Nash mapping $H  : W \times \R^d \to \R^b$
with the following property:

\vspace{1mm}

Let 
$\tau : (W \times \R^d) \times (\R \times \R^d) \times \R^b \to W \times \R^d$ 
and
$\pi : (W \times \R^d) \times (\R \times \R^d) \times \R^b \to \R \times \R^d$
be the canonical projections, 
let $G = (\Delta,H) : W \times \R^d \to (\R \times \R^d) \times \R^b$ be
the Nash mapping defined by $G(x,t) = (\Delta (x,t),H(x,t))$,
and let $X$ be the Zariski closure of graph $G$.
Then there is a union $L$ of some connected components of $X$ with
$\dim L = \dim X$ and $L \subset Reg(X)$ such that 
$\tau |_L : L \to W \times \R^d$ is a $t$-level preserving Nash 
diffeomorphism and $(\pi |_L) \circ (\tau |_L)^{-1} = \Delta$.
\end{thm}

\noindent Therefore we can regard the Nash family of Nash mappings
$\Delta : W \times \R^d \to \R \times \R^d$ between nonsingular,
algebraic varieties as the restriction of some canonical projection
$\pi$ to a union of connected components $L$ of an algebraic variety
$X$.
In this way, we can reduce the Nash case to the algebraic case. 

We next consider the projectivisation so that the mappings $\pi$ and 
$q$ between projectivisated spaces are proper. 
Let us denote by $\R P_{[r]}$ the product
of $r$ real projective lines $\R P^1 \times \cdots \times \R P^1$.
We recall that $W$ is an algebraic variety in $\R^m$.
We denote by $\hat{X}$ the projectivisation of $X$ in
$(\R P_{[m]} \times \R P_{[d]}) \times (\R P^1 \times 
\R P_{[d]}) \times \R P_{[b]}$.
In this paper we regard $\R^r \subset \R P_{[r]}$ and 
$L \subset X \subset \hat{X}$. 

Let 
$$
\hat{\pi} : (\R P_{[m]} \times \R P_{[d]}) \times 
(\R P^1 \times \R P_{[d]}) \times \R P_{[b]} \to
\R P^1 \times \R P_{[d]}
$$ 
and
$\hat{q} : \R P^1 \times \R P_{[d]} \to \R P_{[d]}$ 
be the canonical projections.
We set $\hat{\Pi} := \hat{\pi} |_{\hat{X}}$. 
Then we consider a composite of projections:
\begin{equation}\label{cprojection}
\begin{CD}
(\R P_{[m]} \times \R P_{[d]}) \times (\R P^1 \times \R P_{[d]}) 
\times \R P_{[b]} \supset \hat{X}
@>\text{$\hat{\Pi}$}>> \R P^1 \times \R P_{[d]}
@>\text{$\hat{q}$}>> \R P_{[d]}.
\end{CD}
\end{equation} 

Note that $\hat{X}$ is an algebraic variety and $\hat{\Pi}$ is a family 
of polynomial functions defined on $\hat{X}$ with the parameter space 
$\R P_{[d]}$.
Using a similar argument to Fukuda \cite{fukuda1, fukuda2} 
on $(a_{f_t})$-regular stratification 
with Lojasiewicz's Semialgebraic Triangulation Theorem and 
Fukuda's Lemma for a stratified mapping, 
we have the following assertion. 

\begin{ass}\label{prestratification}
There exist finite Whitney stratifications of Nash class 
$\mathcal{S}(\hat{X})$ and 
$\mathcal{S}(\R P^1 \times \R P_{[d]})$ of $\hat{X}$ and
$\R P^1 \times \R P_{[d]}$, respectively, and a finite subdivision 
of $\R P_{[d]}$ into Nash open simplices 
$\R P_{[d]} = R_1 \cup R_2 \cup \cdots \cup R_e$ 
which satisfy the following.

(1) Let $\mathcal{S}(\R P_{[d]}) = \{ R_1, R_2, \cdots , R_e\}$. 
Then the projections
$$
\hat{\Pi} : (\hat{X},\mathcal{S}(\hat{X}))  \to 
(\R P^1 \times \R P_{[d]},\mathcal{S}(\R P^1 \times \R P_{[d]})),
$$ 
$$\hat{q} : (\R P^1 \times \R P_{[d]},\mathcal{S}(\R P^1 \times \R P_{[d]})) 
\to (\R P_{[d]},\mathcal{S}(\R P_{[d]}))$$ 
are proper stratified mappings. 

(2) For any $t \in R_i$, $1 \le i \le e$, the stratified mapping 
$$
\hat{\Pi}_t : 
(\hat{X}_t,\mathcal{S}(\hat{X})_t) \to 
((\R P^1 \times \R P_{[d]})_t,\mathcal{S}(\R P^1 \times \R P_{[d]})_t)
$$ 
is ($a_{\hat{\Pi}_t}$)-regular.

(3) The stratification $\mathcal{S}(\hat{X})$ is compatible with 
$L \subset X$ and $(q \circ \Pi)^{-1}(R_i)$, $1 \le i \le e$,
and the stratification $\mathcal{S}(\R P^1 \times \R P_{[d]})$ 
is compatible with $\hat{\Pi}(\hat{X})$, $\hat{\Pi}(L) = \pi(L)$, $\R \times \R^d$ 
and $q^{-1}(R_i) = R_i \times \R P_{[d]}$, $1 \le i \le e$.

(4) The stratification $\mathcal{S}(\R P_{[d]})$ is compatible with $\R^d$. 
(Therefore let us assume that $\R^d = R_1 \cup R_2 \cup \cdots \cup R_c$ 
for $c < e$ after this.) 
\end{ass}

We set $\hat{X}^{(i)} := \hat{X} \cap (q \circ \hat{\Pi})^{-1}(R_i)$, 
$\mathcal{S}(\hat{X}^{(i)}) := \{ (q \circ \hat{\Pi})^{-1}(R_i) \cap W \ | \  
W \in \mathcal{S}(\hat{X}) \}$, 
$\hat{\Pi}^{(i)} := \hat{\Pi} |_{\hat{X}^{(i)}}$, 
$\mathcal{S}(\R P^1 \times R_i) : = \{ 
\hat{q}^{-1}(R_i) \cap U \ | \ U \in \mathcal{S}(\R P^1 \times \R P_{[d]}) \}$ 
and $\hat{q}^{(i)} := \hat{q} |_{\R P^1 \times R_i}$, $1 \le i \le c$. 
Let us apply Lemma \ref{SATlemma} to the composit of stratified mappigs 
\begin{equation}\label{cstratified}
\begin{CD}
(\hat{X}^{(i)},\mathcal{S}(\hat{X}^{(i)}))
@>\text{$\hat{\Pi}^{(i)}$}>> (\R P^1\times R_i,\mathcal{S}(\R P^1 \times R_i))
@>\text{$\hat{q}^{(i)}$}>> (R_i,\{ R_i\}),
\end{CD}
\end{equation} 
$1 \le i \le c$.
Then, subdividing $R_i$ into finitely many Nash open simplices if necessary,  
the stratified mapping
$$
\hat{\Pi}^{(i)} : (\hat{X}^{(i)},\mathcal{S}(\hat{X}^{(i)})) \to 
(\R P^1\times R_i,\mathcal{S}(\R P^1 \times R_i))
$$
is semialgebraically trivial over $R_i$, $1 \le i \le c$.
Let $L_i : = L \cap (q \circ \hat{\Pi})^{-1}(R_i)$.
By construction, $\mathcal{S}(\hat{X}^{(i)})$ is compatible with $L_i$ 
and $\mathcal{S}(\R P^1 \times R_i))$ is compatible with 
$\hat{\Pi}(L_i) = \pi(L_i)$ and $\R \times R_i$, $1 \le i \le c$. 
Therefore $\pi |_{L_i} : L_i \to \R \times R_i$ is semialgebraically 
trivial over $R_i$, $1 \le i \le c$. 
It follows that there exists a finite subdivision of $J$ into Nash open 
simplices $J = Q_1 \cup \cdots \cup Q_u$ such that 
$$
\{ f_t : M \to \R \ | \ t \in Q_i \} 
$$
is semialgebraically $\mathcal{RL}$ trivial over each $Q_i$.

This completes the proof Theorem \ref{finitenessNF}.
\end{proof}


\end{document}